\documentclass [a4paper,11pt]{amsart}
\usepackage{graphicx}
\usepackage[all]{xy}

\usepackage{amssymb,amscd,amsthm,amsopn,amsmath}
\usepackage{mathrsfs}

\def\C{\mathbb C}

\def\N{\mathbb N}
\def\dim{\operatorname{dim}}
\def\gdc{\operatorname{gdc}}

\def\codim{\operatorname{codim}}

\usepackage{graphicx}
\usepackage[ansinew]{inputenc}    
\usepackage[all]{xy}
\usepackage{amssymb,amscd}
\usepackage{color}
\usepackage{mathrsfs}

\newtheorem{thm}{Theorem}[section]
\newtheorem{cor}[thm]{Corollary}
\newtheorem{teo}[thm]{Theorem}
\newtheorem{lem}[thm]{Lemma}
\newtheorem{prop}[thm]{Proposition}

\theoremstyle{definition}
\newtheorem{exam}[thm]{Example}
\newtheorem{defi}[thm]{Definition}

\def\A{\mathscr A}
\def\C{\mathbb C}

\def\N{\mathbb N}

\def\O{\mathcal O}

\def\dim{\operatorname{dim}}

\def\codim{\operatorname{codim}}

\def\Derlog{\operatorname{Derlog}}

\begin{document}
\title[Image Milnor number and $\A_e$-codimension]{Image Milnor number and $\A_e$-codimension for maps between weighted homogeneous irreducible curves}

\author{D. A. H. Ament, J. J. Nu\~no-Ballesteros, J. N. Tomazella}

\address{Departamento de Matem\'atica, Universidade Federal de S\~ao Carlos, Caixa Postal 676,
13560-905, S\~ao Carlos, SP, BRAZIL}

\email{daianehenrique@dm.ufscar.br}

\address{Departament de Matem\`atiques,
Universitat de Val\`encia, Campus de Burjassot, 46100 Burjassot
SPAIN}

\email{Juan.Nuno@uv.es}

\address{Departamento de Matem\'atica, Universidade Federal de S\~ao Carlos, Caixa Postal 676,
13560-905, S\~ao Carlos, SP, BRAZIL}

\email{tomazella@dm.ufscar.br}

\thanks{The first author has been supported by CAPES. The second author has been partially supported by DGICYT Grant MTM2015--64013--P. The third author is partially supported by CNPq Grant 309626/2014-5 and FAPESP Grant 2016/04740-7.}


\begin{abstract} Let $(X,0)\subset (\C^n,0)$ be an irreducible weighted homogeneous singularity curve and let $f:(X,0)\to(\C^2,0)$ be a map germ finite, one-to-one and weighted homogeneous with the same weights of $(X,0)$. We show that $\A_e$-$\codim(X,f)=\mu_I(f)$, where $\A_e$-$\codim(X,f)$ is the $\A_e$-codimension, i.e., the minimum number of parameters in a versal deformation and $\mu_I(f)$ is the image Milnor number, i.e., the number of vanishing cycles in the image of a stabilisation of $f$. \end{abstract}

\maketitle

\section{Introduction}

Let $\alpha:(\C,S) \to(\C^{2},0)$ be a map germ of finite $\A_e$-codimension. D. Mond shows in \cite{M2} that the image Milnor number of $\alpha$, $\mu_I(\alpha)$, is determined by the number of branches $r(X,0)$ of the curve and the number $\delta(X,0)$ of nodes appearing in a stable perturbation of $\alpha$:
$$
\mu_I(\alpha)=\delta(X,0)-r(X,0)+1.
$$

D. Mond also proves in \cite{M2} a relation between the image Milnor number and the $\A_{e}$-codimension.

\begin{teo}[\cite{M2}]\label{Mond2}
Let $\alpha: (\C,S) \to (\C^{2},0)$ be a map germ of finite $\A_{e}$-codimension. Then, 
$$\A_{e}\mbox{-}\codim(\alpha) \leq \mu_{I}(\alpha),
$$
with equality if $\alpha$ is weighted homogeneous.
\end{teo}

Inspired by the previous inequality, the first and second authors consider in \cite{DJ} map germs $f:(X,0)\to(\C^2,0)$, where $(X,0)$ is a plane curve. They define the image Milnor number of $f$ and  obtain the similar inequality.

\begin{teo}[\cite{DJ}] Let $(X,0)$ be a plane curve and let $f:(X,0)\to(\C^2,0)$ be a finite map germ of degree 1 onto its image $(Y,0)$. 
Then, $$\A_e\text{-}\codim(X,f) \leq \mu_{I}(f),$$
with equality if and only if $(Y,0)$ is weighted homogeneous.
\end{teo}

Furthermore, following the proof of this result, it is possible obtain an equality,
$$\A_{e}\mbox{-}\codim(X,f) + \mu(Y,0) - \tau(Y,0) = \mu_{I}(f).$$

In this work, we consider map germs $f:(X,0)\to(\C^2,0)$, where $(X,0)$ is an isolated complete intersection singularity (ICIS) of dimension one, then we can consider the $\A_e$-codimension, $\A_e$-$\codim(X,f)$, in the sense of \cite{MM}.
On the other hand, we define the image Milnor number in this case as the number of vanishing cycles in the image of a stable perturbation $f_s:X_s\to B_{\epsilon}$, where $B_{\epsilon}$ is a ball of radius $\epsilon$ centered at the origin in $\C^2$. Here, stable means that $X_s$ is smooth and $f_s$ is stable in the usual sense. 
We show that $\A_e$-$\codim(X,f)= \mu_I(f)$, if $(X,0)$ is irreducible and both $(X,0)$ and $f$ are weighted homogeneous (with the same weights).

Furthermore, if we consider $(X,0)$ parametrized by a map $\alpha:(\C,0)\to(\C^n,0)$, we have also the following equalities.
$$
\mu_{I}(f)=\delta(X,0)+\mu_{I}(f\circ\alpha)
$$
and
$$
\A_e\mbox{-}\codim(f)=\A_e\mbox{-}\codim(f\circ\alpha)-\frac{1}{n-1}\A_e\mbox{-}\codim(\alpha)
$$


\section{Maps on singular varieties}

In this section, we give the basic definitions we will use in the work. 
First, we introduce some notations. We denote $\mathcal{O}_{n}$ the local ring of the analytic functions germs $f:(\C^n,0) \to \C$, $(X,0)\subset (\C^n,0)$ is a germ of analytic variety (possibly with singularities),
$I(X,0)$ is the ideal of $\mathcal O_n$ of functions vanishing on $(X,0)$,  $\O_{X,0}=\O_n/I(X,0)$ is the local ring of $(X,0)$,
$\Theta_n$ is the $\mathcal O_n$-module of the vector fields in $(\C^n,0)$ and $\Theta_{X,0}$ is the $\mathcal{O}_{X,0}$-module of vector fields tangents on $(X,0)$.

We refer to \cite{MM} for the general definition of $\A_e$-codimension for analytic map germs $f:(X,0)\to(\C^p,0)$. The $\A_e$-codimension is equal to
$$
\A_{e}\mbox{-}\codim(f)=\dim_\C\frac{\Theta(f)}{tf(\Theta_{X,0})+ \omega f(\Theta_{p})},
$$
where $\Theta(f)$ is the $\mathcal{O}_{X,0}$-module of vector fields along $f$, that is, holomorphic germs $\xi:(X,0)\to T\C^p$ such that $\pi\circ\xi=f$ (where $\pi:T\C^p\to\C^p$ is the canonical projection). The map
$tf:\Theta_{X,0}\to\Theta(f)$ is the morphism of $\mathcal{O}_{X,0}$-modules given by $tf(\xi)=df\circ \xi$
and $\omega f:\Theta_p\to\Theta(f)$  is the morphism of $\mathcal{O}_{p}$-modules given by $\omega f(\eta)=\eta\circ f$ (where $\Theta(f)$ is considered as an $\mathcal{O}_{p}$-module via $f^*:\O_p\to\O_{X,0}$). We can denote $T\A_{e} f=tf(\Theta_{X,0})+ \omega f(\Theta_{p})$.

We say that $f$ is \emph{$\mathscr A$-finite} if this codimension is finite. We say that $f$ has \emph{finite singularity type} if
$$
\dim_\C\frac{\Theta(f)}{tf(\Theta_{X,0})+(f^*m_p)\Theta_{p}}<\infty.
$$

In the case that $(X,0)$ is an ICIS and $f$ has finite singularity type, we have the following important result due to Mond and Montaldi \cite{MM}:

\begin{thm}\label{estabilizacao} Let $(X,0)$ be an ICIS and assume $f:(X,0)\to(\C^p,0)$ has finite singularity type. The minimal number of parameters in a versal unfolding of $f$ is equal to the number
$$
\mathscr A_e\text{-}\codim(X,f):=\mathscr A_e\text{-}\codim(f)+\tau(X,0),
$$
where $\tau(X,0)$ is the Tjurina number of $(X,0)$, that is, the minimal number of parameters in a versal deformation of $(X,0)$.
\end{thm}

As a corollary of this theorem, we have the following consequences. The second one is a generalization of the Mather-Gaffney criterion (see \cite{wall}).

\begin{cor}\label{infinitesimal} Let $(X,0)$ be an ICIS and assume $f:(X,0)\to(\C^p,0)$ has finite singularity type. 
\begin{enumerate}
\item $f$ is stable if and only if $X$ is smooth and $f$ is stable in the usual sense.
\item $f$ is $\mathscr A$-finite if and only if $f$ has isolated instability (i.e., there is a representative $f:X\to B_\epsilon$ such that for any $y\in B_\epsilon\setminus\{0\}$, the multigerm of $f$ at $f^{-1}(y)\cap S$ is stable).
\end{enumerate}
\end{cor}

Another consequence of the theorem is the existence of stabilizations, at least in the range of Mather's nice dimensions. Given an analytic map germ $f:(X,0)\to(\C^p,0)$, a \emph{stabilization} is a 1-parameter unfolding $F:(\mathcal X,0)\to (\C\times\C^p)$ with the property that for all $s\ne0$ small enough, the map $f_s:X_s\to B_\epsilon$ is stable, where $B_\epsilon$ is a ball of radius $\epsilon$ in $\C^p$.
By Theorem \ref{estabilizacao}, if $f:(X,0)\to(\C^p,0)$ is $\mathscr A$-finite, a stabilization of $f$ exists if $(r,p)$ are nice dimensions in the Mather's sense ($r=\dim(X,0)$). 

We will consider $f:(X,0)\to(\C^2,0)$, where $(X,0)\subset(\C^n,0)$ is a curve ICIS. 
Observe that by Corollary \ref{infinitesimal}, $f$ is stable if and only if $X$ is smooth and $f$ is an immersion with only transverse double points, called nodes. 
As a consequence, $f$ is $\mathscr A$-finite if and only if $f$ is finite and has degree one onto its image. 

\section{Weighted homogeneous maps and varieties}

We introduce the definition of weighted homogeneous for variety germs and map germs. 
We fix positive integer numbers $w_1,\cdots,w_n$ such that $\gdc(w_1,\cdots,w_n)=1$. 
Given $h \in \mathcal{O}_n$, we say that $h$ is weighted homogeneous of type $(w_1,\cdots,w_n;d)$ if it satisfies
$$h(t^{w_1}x_1,\cdots,t^{w_n}x_n)=t^d h(x_1,\cdots,x_n), \quad \forall x \in \C^n, \quad \forall t \in \C.$$
In this case, we call $(w_1,\cdots,w_n)$ the weights of $h$ and $d$ the weighted degree of $h$.
Let $(X,0)\subset(\C^n,0)$ be a variety germ, which is the zero set of an ideal $I \subset \mathcal{O}_n$, we say that $(X,0)$ is weighted homogeneous of type $(w_1,\cdots,w_n;d_1,\cdots,d_m)$ 
if $I$ can be generated by weighted homogeneous map germs $\phi_1,\cdots,\phi_m$ where each $\phi_i$ is weighted homogeneous of type $(w_1,\cdots,w_n,d_i)$, $i=1,\cdots,n$.
Finally, let $f:(X,0)\to (\C^p,0)$ be a map germ, with $(X,0)\subset (\C^n,0)$ a weighted homogeneous variety germ of type $(w_1,\cdots,w_n;d_1,\cdots,d_m)$. 
We say that $f$ is weighted homogeneous if $f$  is the restriction of a map germ $\tilde{f}=(f_1,\cdots,f_p):(\C^n,0)\to(\C^p,0)$ where each $f_j \in \mathcal{O}_n$ is weighted homogeneous
and we say that $f$ is {\it consistent} with $(X,0)$ if each $f_j$ is weighted homogeneous with the weights $(w_1,\cdots,w_n)$, we observe that the weighted degree may be different.

It follows from \cite{NOT} that if $(X,0)$ is a curve ICIS weighted homogeneous of type $(w_1,\cdots,w_n;d_1,\cdots,d_{n-1})$, then the Milnor number satisfies
$$\mu(X,0)=\frac{d_1\cdots d_{n-1}\left(d_1 + \cdots+ d_{n-1} - w_1 - \cdots - w_n\right)}{w_1 \cdots w_n}+1.$$
Moreover, if $(X,0)$ is irreducible, then $\frac{d_1\cdots d_{n-1}}{w_1 \cdots w_n}=1$, thus the above equality is given by 
$$\mu(X,0)=d_1 + \cdots + d_{n-1} - w_1 - \cdots - w_n + 1.$$

Wahl \cite{wahl} shows how to compute the generators of $\Theta_{X,0}$ when $(X,0)$ is a weighted homogeneous ICIS. 
In particular, if we suppose $(X,0)=h^{-1}(0)$ a curve, where $h=(h_1, \cdots, h_{n-1})$, then $\Theta_{X,0}$  is generated by vector fields $h_i \frac{\partial}{\partial x_j}$ with $i=1,\cdots,n-1$ and $j=1, \cdots, n$, the Euler field $\epsilon= (w_1x_1,\cdots,w_nx_n)$ and $\mathcal H = (\xi_1,\cdots,\xi_n),$
where
\begin{align*}
\mathcal H &= \left| \begin{array}{ccc}
\frac{\partial}{\partial x_1}&\cdots & \frac{\partial}{\partial x_n}  \\
\frac{\partial h_1}{\partial x_1}&\cdots & \frac{\partial h_{1}}{\partial x_n}  \\
\vdots                                         & \ddots & \vdots \\
\frac{\partial h_{n-1}}{\partial x_1}&\cdots & \frac{\partial h_{n-1}}{\partial x_n}\\
\end{array} \right| \\
& = (-1)^{1+1} \left| \begin{array}{cccc}
\frac{\partial \widehat{h}_1}{\partial x_1} &\frac{\partial h_1}{\partial x_2}&\cdots & \frac{\partial h_{1}}{\partial x_n}  \\
\vdots                   & \vdots                      & \ddots & \vdots \\
\frac{\partial h_{n-1}}{\partial x_1} & \frac{\partial h_{n-1}}{\partial x_2}&\cdots & \frac{\partial h_{n-1}}{\partial x_n}\\
\end{array} \right| \frac{\partial}{\partial x_1} + \cdots \\
& \quad +(-1)^{n+1} \left| \begin{array}{cccc}
\frac{\partial h_1}{\partial x_1}&\cdots & \frac{\partial h_{1}}{\partial x_{n-1}}& \frac{\partial \widehat{h}_{1}}{\partial x_{n}}  \\
\vdots                                         & \ddots & \vdots & \vdots\\
\frac{\partial h_{n-1}}{\partial x_1}&\cdots & \frac{\partial h_{n-1}}{\partial x_{n-1}}& \frac{\partial h_{n-1}}{\partial x_{n}}\\
\end{array} \right| \frac{\partial}{\partial x_{n}}\\
&=\xi_1 \frac{\partial}{\partial x_1} + \cdots + \xi_n \frac{\partial}{\partial x_n},
\end{align*}
we use \ $\widehat{}$ \ to indicate that we should exclude this column. 

Let $J_1,J_2\in\O_{X,0}$ be defined as
\begin{align*}
J_i &= \left| \begin{array}{ccc}
\frac{\partial h_1}{\partial x_1} &\cdots & \frac{\partial h_{1}}{\partial x_n}\\
\vdots                                         & \ddots & \vdots                                                 \\
\frac{\partial h_{n-1}}{\partial x_1}   &\cdots & \frac{\partial h_{n-1}}{\partial x_n} \\
\frac{\partial f_i}{\partial x_1}   & \cdots &  \frac{\partial f_i}{\partial x_n} \\ \end{array} \right|  =(-1)^{n-1}\left(\frac{\partial f_i}{\partial x_1} \xi_1+  \cdots+
\frac{\partial f_i}{\partial x_n}\xi_n\right),\\
\end{align*}
with $i=1,2$.

We have the matrix 
$$df= \left( \begin{array}{ccc}
\frac{\partial f_1}{\partial {x}_1}&\cdots & \frac{\partial f_1}{\partial {x}_{n}}  \\
\frac{\partial f_2}{\partial x_1}&\cdots & \frac{\partial f_2}{\partial x_n}  \\
\end{array} \right).
$$

Thus, $tf(\Theta_{X,0})$ is generated by vector fields
$$\left(\frac{\partial f_1}{\partial {x}_1} h_1, \frac{\partial f_2}{\partial {x}_1}h_1\right),\cdots,\left(\frac{\partial f_1}{\partial {x}_n}h_1,\frac{\partial f_2}{\partial {x}_n}h_1\right),\cdots,$$
$$\left(\frac{\partial f_1}{\partial {x}_1}h_{n-1},\frac{\partial f_2}{\partial {x}_1}h_{n-1}\right),\cdots,\left(\frac{\partial f_1}{\partial {x}_n}h_{n-1},\frac{\partial f_2}{\partial {x}_n}h_{n-1}\right),$$
\begin{align*}
df \circ \epsilon & = \left(w_1x_1\frac{\partial f_1}{\partial {x}_1}+\cdots+w_nx_n\frac{\partial f_1}{\partial {x}_{n}},w_1x_1\frac{\partial f_2}{\partial {x}_1}+\cdots+w_nx_n\frac{\partial f_2}{\partial {x}_{n}}\right)\\
&=(l_1 \beta_ 1 f_1,l_2 \beta_2 f_2)\\
\end{align*}
and
\begin{align*}
df\circ\mathcal H& = \left(\xi_1 \frac{\partial f_1}{\partial x_1}+ \cdots+  \xi_n \frac{\partial f_1}{\partial x_{n}} , \xi_1 \frac{\partial f_2}{\partial x_1}+ \cdots+  \xi_n \frac{\partial f_2}{\partial x_{n}}\right)\\
&= (-1)^{n+1} \left(J_1,J_2\right).
\end{align*}



\section{Maps between curves}

Let $(X,0)\subset(\C^n,0)$ be an irreducible curve ICIS, weighted homogeneous of type $(w_1,\cdots,w_n;d_1,\cdots,d_{n-1})$. 
Apart from the case that $(X,0)$ is smooth (which is a trivial case), it follows that $(X,0)$ can be parametrized by a monomial map 
$\alpha:(\C,0)\to(\C^n,0)$ of the form
$$ \alpha(t)=(\alpha_1 t^{w_1},\cdots,\alpha_n t^{w_n}),\quad (\alpha_1,\cdots,\alpha_n)\in\C^n \setminus\{0\}.$$

The local ring of $(X,0)$ is $\O_{X,0}=\C\{t^{w_1},\cdots,t^{w_n}\}$, which is a subring of $\O_1$. We denote by $\Gamma_X$ the associated semigroup. Since $(X,0)$ is irreducible, $\Gamma_X$ has a conductor $c$, which satisfies that $c-1 \notin \Gamma_X$ and if $n\in \N$, with $n \geq c$, then $n\in \Gamma_X$. We have that $\Gamma_X$ is symmetric, because $(X,0)$ is an ICIS, thus $\#\N\setminus \Gamma_X=c/2$. 

The main invariant of $(X,0)$ is the delta invariant, which is in this case 
$$\delta_X=\dim_{\C}\frac{\C\left \{ t \right \}}{\C\left \{ t^{w_1},\cdots,t^{w_n} \right \}}=\#\N\setminus \Gamma_X.$$
Since $(X,0)$ is irreducible, the Milnor number $\mu(X,0)$, which to simplify the notation we will denote by $\mu_X$, is $\mu_X=2\delta_X=c$, by Milnor's formula. 
For more details, see for instance \cite{HeHer}.

Let $f:(X,0)\to(\C^2,0)$ be a finite map germ of degree 1 onto its image $(Y,0)$ and $f$ is consistent with $(X,0)$. We denote by $l_1,l_2$ the weighted degrees of $f_1,f_2$, respectively. 
Then, $(Y,0)$ is also a weighted homogeneous irreducible plane curve with weights $(l_1, l_2)$ and it admits a parametrization of the form
$$f(\alpha(t))=(\beta_1 t^{l_1},\beta_2 t^{l_2}),\quad \beta_1,\beta_2\in\C\setminus\{0\}.$$
We have $\O_{Y,0}=\C\{t^{l_1},t^{l_2}\}\subset \O_{X,0}$,  $\Gamma_Y=\langle l_1,l_2\rangle\subset \Gamma_X$ and $\delta_Y=(l_1-1)(l_2-1)/2$. Finally, we assume that $(Y,0)$ has defining equation $g(u,v)=0$, where $g\in\C\{u,v\}$ is weighted homogeneous.
Thus, $g(u,v)=\beta_{2}^{l_1}u^{l_2}-\beta_{1}^{l_2}v^{l_1}$ is weighted homogeneous  of type $(l_1, l_2;l_1l_2)$.

We consider $f=i\circ \bar{f}$, where $\bar{f}:(X,0)\to(Y,0)$ is the restriction of $f$ and $i:(Y,0)\to(\C^2,0)$ is the inclusion map. We have,
$$
J_{g}\mathcal{O}_{X,0} = \left \{ a \left(\frac{\partial g}{\partial u}\circ f \right)+ b \left(\frac{\partial g}{\partial v}\circ f \right) ; \  a, b \in \mathcal{O}_{X,0} \right \},
$$
and

$
J_{g}\mathcal{O}_{Y,0} = \left \{\left(a \circ \bar{f} \right) \left( \left(\frac{\partial g}{\partial u} \circ i \right) \circ  \bar{f} \right) + \left( b \circ \bar{f} \right) \left( \left( \frac{\partial g}{\partial v}  \circ i \right) \circ \bar{f} \right) ; \  a, b \in \mathcal{O}_{Y,0} \right \},
$
where $J_g$ be the Jacobian ideal of $g$, thus $J_{g}\mathcal{O}_{Y,0}  \subset J_{g}\mathcal{O}_{X,0}$.

We consider the evaluation map
$$ev: \frac{\Theta(f)}{tf(\Theta_{X,0})+ wf(\Theta_{2})} \longrightarrow \frac{J_{g}\mathcal{O}_{X,0}}{J_{g}\mathcal{O}_{Y,0}}$$
given by $ev([\xi]) = [\xi(g)]$. 

\begin{lem}\label{sobrejetora}
The evaluation map is well-defined and surjective, hence
$$ \A_{e}\mbox{-}\codim(f) = \mbox{\rm dim}_{\C} \ker(ev) + \mbox{\rm dim}_{\C} \frac{J_{g}\mathcal{O}_{X,0}}{J_{g}\mathcal{O}_{Y,0}}.$$
\end{lem}
\begin{proof}
Let us see that $ev$ is well-defined. 
If $\xi \in tf(\Theta_{X,0})$, then $\xi = tf(\eta)$, for some $\eta \in \Theta_{X,0}$, hence $\eta(l)=0$ in $\mathcal{O}_{X,0}$, for all $l\in I(X,0)$.
Since $g \circ \tilde f = \lambda_1 h_1 + \cdots \lambda_{n-1} h_{n-1}$, with $\lambda_i \in \mathcal{O}_{n}$, $i=1,\cdots,n$, we have 
$$ 
\xi(g) =tf(\eta)(g)=(df\circ \eta)(g) = \eta(g \circ \tilde{f})= \eta(\lambda_1 h_1 + \cdots \lambda_{n-1} h_{n-1})=0.$$

If $\xi \in \omega f(\Theta_{2})$, then $\xi = \zeta \circ f$, for some $\zeta \in \Theta_{2}$. Write  $\zeta = \tilde{a} \frac{\partial}{\partial u} + \tilde{b} \frac{\partial}{\partial v}$, with $\tilde a,\tilde b\in\O_2$. Thus, 
$$\xi=\zeta \circ f =  \left( \tilde{a} \circ f \right) \frac{\partial}{\partial u} + \left(\tilde{b} \circ f \right) \frac{\partial}{\partial v}.
$$

Then, 
\begin{align*}
\xi(g)& =\left( \tilde{a} \circ f \right) \frac{\partial g}{\partial u} \circ f + \left(\tilde{b} \circ f \right) \frac{\partial g}{\partial v} \circ f \\
& = \left( \left(\tilde{a} \circ i \right) \circ \bar{f} \right) \left( \left(\frac{\partial g}{\partial u} \circ i \right) \circ  \bar{f} \right) + \left( \left(\tilde{b} \circ i \right) \circ \bar{f} \right) \left( \left( \frac{\partial g}{\partial v}  \circ i \right) \circ \bar{f} \right),
\end{align*}
with $\tilde{a} \circ i ,\tilde{b} \circ i \in \mathcal{O}_{Y,0}$. Hence, $\xi(g) \in J_{g}\mathcal{O}_{Y,0}$.

If $\xi= a \frac{\partial}{\partial u} + b \frac{\partial}{\partial v}$, with $a$, $b \in \mathcal{O}_{X,0}$, then
$\xi(g) = a \frac{\partial g}{\partial u} \circ f + b \frac{\partial g}{\partial v} \circ f$, hence $ev$ is surjective.
\end{proof}

The evaluation map is not injective in general, so we need to know its kernel in order to compute the $\A_e$-codimension of $f$.
To this we need the following result.

\begin{lem} \label{lemmalambda} With the above notation, we have
$$ 
\frac{\partial g}{\partial u} \circ f  = C \lambda_f J_2  \quad \mbox{and} \quad \frac{\partial g}{\partial v} \circ f = - C  \lambda_f  J_1,
$$
when $\lambda_f=t^{\mu_Y-\mu_X}\in \O_{X,0}$.
\end{lem}
\begin{proof}
By definitions, we obtain
$$
\frac{\partial g}{\partial u} \circ f  =  l_2 \beta_{2}^{l_1} \beta_1^{l_2-1}t^{l_1(l_2-1)} \quad \mbox{and}  \quad \frac{\partial g}{\partial v} \circ f  =  -l_1 \beta_{1}^{l_2} \beta_2^{l_1-1} t^{l_2(l_1-1)}.  
$$

We also have
\begin{align*}
J_j&=\left| \begin{array}{ccc}
t^{d_1-w_1}\frac{\partial h_1}{\partial x_1}(\alpha_1,\cdots,\alpha_n) &\cdots & t^{d_1-w_n}\frac{\partial h_{1}}{\partial x_n}(\alpha_1,\cdots,\alpha_n)\\
\vdots                                         & \ddots & \vdots                                                 \\
t^{d_{n-1}-w_1}\frac{\partial h_{n-1}}{\partial x_1} (\alpha_1,\cdots,\alpha_n)  &\cdots & t^{d_{n-1}-w_n}\frac{\partial h_{n-1}}{\partial x_n}(\alpha_1,\cdots,\alpha_n) \\
t^{l_j-w_1}\frac{\partial f_j}{\partial x_1} (\alpha_1,\cdots,\alpha_n)  & \cdots &  t^{l_j-w_n}\frac{\partial f_j}{\partial x_n}(\alpha_1,\cdots,\alpha_n) \\ \end{array} \right|.
\end{align*}
We denote $\frac{\partial h_i}{\partial x_k}(\alpha_1,\cdots,\alpha_n)=\gamma_{i,k}$ with $i=1,\cdots,n-1$, $k=1,\cdots,n$ and $\frac{\partial f_j}{\partial x_k} (\alpha_1,\cdots,\alpha_n)=\zeta_{j,k}$ with $j=1,2$ and $k=1,\cdots,n$.

Furthermore, since $h_1, \cdots, h_{n-1}, f_1, f_2$ are weighted homogeneous with weights $w_1, \cdots, w_n$, they satisfies:
$$d_ih_i=w_1 x_1\frac{\partial h_i}{\partial x_1}+\cdots+w_n x_n\frac{\partial h_i}{\partial x_n},$$
$$l_jf_j=w_1 x_1\frac{\partial f_j}{\partial x_1}+\cdots+w_n x_n\frac{\partial f_j}{\partial x_n},$$
with $i=1,\cdots,n-1$ and $j=1,2$. Thus,
$$w_1 \alpha_1t^{w_1}t^{d_i-w_1}\gamma_{i,1}+\cdots+w_n \alpha_n t^{w_n} t^{d_i-w_n}\gamma_{i,n}=0,$$
$$w_1 \alpha_1t^{w_1}t^{l_j-w_1}\zeta_{j,1}+\cdots+w_n \alpha_n t^{w_n} t^{l_j-w_n}\zeta_{j,n}=l_j\beta_j t^{l_j},$$
with $i=1,\cdots,n-1$ and $j=1,2$. 

Rewriting in matrix form
$$\left( \begin{array}{ccc}
t^{d_1-w_1}\gamma_{1,1} &\cdots & t^{d_1-w_n}\gamma_{1,n}\\
\vdots                                         & \ddots & \vdots                                                 \\
t^{d_{n-1}-w_1}\gamma_{n-1,1} &\cdots & t^{d_{n-1}-w_n}\gamma_{n-1,n}\\
t^{l_j-w_1}\zeta_{j,1} & \cdots &  t^{l_j-w_n}\zeta_{j,n} \\ \end{array} \right)\left( \begin{array}{c}
 w_1 \alpha_1t^{w_1}\\
\vdots                  \\
w_n \alpha_n t^{w_n} \\ \end{array} \right)= \left( \begin{array}{c}
 0\\
\vdots                  \\
0\\
l_j\beta_j t^{l_j} \\
 \end{array} \right).
$$

By Cramer's Rule,
$$w_1 \alpha_1 t^{w_1} J_j= \left| \begin{array}{cccc}
 0 & t^{d_1-w_2}\gamma_{1,2} &\cdots & t^{d_1-w_n}\gamma_{1,n}\\
\vdots   & \vdots                                      & \ddots & \vdots                                                 \\
0& t^{d_{n-1}-w_2}\gamma_{n-1,2} &\cdots & t^{d_{n-1}-w_n}\gamma_{n-1,n}\\
l_j\beta_j t^{l_j}& t^{l_j-w_2}\zeta_{j,2} & \cdots &  t^{l_j-w_n}\zeta_{j,n} \\ \end{array} \right|.
$$

Thus,
\begin{align*}
J_j&=(-1)^{1+n} \frac{l_j\beta_j}{w_1\alpha_1} t^{l_j-w_1}\left| \begin{array}{ccc}
  t^{d_1-w_2}\gamma_{1,2} &\cdots & t^{d_1-w_n}\gamma_{1,n}\\
\vdots                                      & \ddots & \vdots                                                 \\
t^{d_{n-1}-w_2}\gamma_{n-1,2} &\cdots & t^{d_{n-1}-w_n}\gamma_{n-1,n}\\  \end{array} \right| \\
&= l_j\beta_j t^{l_j - w_1 + d_1+ \cdots + d_{n-1}- w_2 - \cdots - w_n} A, \\
\end{align*}
with $A\neq 0$.

We obtain,
$$
J_1 = l_1 \beta_1 A  t^{l_1+\mu_X-1} \quad \mbox{and} \quad J_2 = l_2 \beta_2 A t^{l_2+\mu_X-1}
$$

Therefore,
\begin{align*}
\frac{\partial g}{\partial u} \circ f  &=  l_2 \beta_{2}^{l_1} \beta_1^{l_2-1}t^{l_1(l_2-1)}= \frac{ \beta_2^{l_1-1}\beta_1^{l_2-1}}{A} t^{\mu_Y - \mu_X} l_2 \beta_2 A  t^{l_2+\mu_X-1}\\
&= C \lambda_f J_2
\end{align*}
and
\begin{align*}
\frac{\partial g}{\partial v} \circ f & =  - l_1 \beta_{1}^{l_2} \beta_2^{l_1-1}t^{l_2(l_1-1)}=- \frac{ \beta_1^{l_2-1}\beta_2^{l_1-1}}{A} t^{\mu_Y - \mu_X} l_1 \beta_1 A  t^{l_1+\mu_X-1}\\
&= -C \lambda_f J_1.
\end{align*}
\end{proof}

\begin{lem} \label{nucleo} 
We consider the map 
$$\widetilde{ev}: \frac{\Theta(f)}{tf(\Theta_{X,0})} \longrightarrow J_g \O_{X,0},$$
given by $\widetilde{ev}([\xi])=\xi(g)$. Then, $\ker(ev) \cong \ker(\widetilde{ev})$.
\end{lem}
\begin{proof}
We observe that the map $\widetilde{ev}$ is well-defined and it is surjective. 

We define the morphism $\ker(ev)  \to \ker(\widetilde{ev})$ given by $\left[\xi\right] \mapsto [\xi]$, where the classes have to be considered in the respective quotients.

We claim that if $[\xi] \in \ker(ev)$, there exists representative $\xi$ such that
$$ \xi(g) = a \frac{\partial g}{\partial u} \circ f + b \frac{\partial g}{\partial v} \circ f = 0.$$
In fact, if $[\xi] \in \ker(ev)$, we choose a representative $\xi = a \frac{\partial}{\partial u} + b \frac{\partial}{\partial v}$, $a$, $b \in \mathcal{O}_{X,0}$, thus $\xi(g) = a \frac{\partial g}{\partial u} \circ f + b \frac{\partial g}{\partial v} \circ f$.
On the other hand, since $[\xi] \in \ker(ev)$, we have $\xi(g) \in J_{g}\mathcal{O}_{Y,0}$, so,
$$\xi(g) = \left( \tilde{a} \circ i \circ \bar{f} \right) \left( \left(\frac{\partial g}{\partial u} \circ i \right) \circ  \bar{f} \right) + \left( \tilde{b} \circ i \circ \bar{f} \right) \left( \left( \frac{\partial g}{\partial v}  \circ i \right) \circ \bar{f} \right),$$
with $\tilde{a} \circ i, \tilde{b} \circ i  \in \mathcal{O}_{Y,0}$. Therefore,
$$\left(  a - \tilde{a} \circ i \circ \bar{f} \right) \left( \frac{\partial g}{\partial u}  \circ  f \right) + \left(b - \tilde{b}  \circ i \circ \bar{f} \right)  \left( \frac{\partial g}{\partial v}   \circ f \right) = 0.$$
We define $\bar{a} = a - \tilde{a} \circ i \circ \bar{f}$, $\bar{b} = b - \tilde{b} \circ i \circ \bar{f}$ and $\bar{\xi} = \bar{a} \frac{\partial}{\partial u} + \bar{b} \frac{\partial}{\partial v}$, thus 
\begin{align*}
\xi - \bar{\xi}& =  a \frac{\partial}{\partial u} + b \frac{\partial}{\partial v} -  \left( \bar{a} \frac{\partial}{\partial u} + \bar{b} \frac{\partial}{\partial v} \right) = 
\left( \tilde{a} \circ i \circ \bar{f}\right) \frac{\partial}{\partial u} + \left(\tilde{b}  \circ i \circ \bar{f}\right) \frac{\partial }{\partial v}\\
&=  \left(\tilde{a}  \circ f\right) \frac{\partial}{\partial u} +\left( \tilde{b}   \circ f \right) \frac{\partial }{\partial v} =
 \tilde{\zeta} \circ f,
\end{align*}
where $\tilde{\zeta}=\tilde{a} \frac{\partial}{\partial u} + \tilde{b} \frac{\partial }{\partial v}$, with $\tilde{a}, \tilde{b} \in \mathcal{O}_{2}$.
Thus, $\xi - \bar{\xi} \in wf(\Theta_{2})$, i.e., $[\xi]=[\bar{\xi}]$ with $\bar{\xi}(g) = 0$.  This prove the claim.

Now, we will show that the morphism is well-defined. In fact, let $[\xi] = [\tilde{\xi}] \in \ker(ev)$, i.e., $\xi + T\A_{e} f = \tilde{\xi}+T\A_{e} f $, then $\xi - \tilde{\xi} \in T\A_{e} f$, with $\xi(g)=\tilde{\xi}(g)=0$.
We will show that $\xi - \tilde{\xi} \in tf(\Theta_X)$.

Since  $\xi - \tilde{\xi} \in T\A_{e} f$, then  $\xi - \tilde{\xi} = df \circ \eta + \rho \circ f$, $\eta \in \Theta_X$, $\rho \in \Theta_2$, with
$0 = \xi(g) - \tilde{\xi}(g) = df \circ \eta(g) + \rho \circ f(g)$, moreover $df \circ \eta(g) = 0$, thus $\rho \circ f(g)=0$.

We have $\rho =(a,b)$, with $a, b \in \mathcal{O}_{2}$, then, $\rho \circ f = (a \circ f,b \circ f)$ such that
$ (a \circ f) \frac{\partial g}{\partial u} \circ f + (b \circ f) \frac{\partial g}{\partial v}\circ f =0$. Thus,
$$\left( a \frac{\partial g}{\partial u} + b \frac{\partial g}{\partial v}\right)\circ f =0, \  \mbox{in}  \  \mathcal{O}_{Y}$$
i.e.,
$$ a \frac{\partial g}{\partial u} + b \frac{\partial g}{\partial v} = \kappa g, \ \kappa \in \mathcal{O}_2.$$
Since $g$ is weighted homogeneous,
$$ l_1 l_2 g = l_1 u \frac{\partial g}{\partial u}+l_2 v \frac{\partial g}{\partial v},$$
so,
$$l_1 l_2 \left(a \frac{\partial g}{\partial u} + b \frac{\partial g}{\partial v}\right) = l_1 l_2 \kappa  g = \kappa \left(l_1 u \frac{\partial g}{\partial u}+l_2 v \frac{\partial g}{\partial v}\right),$$

$$(l_1 l_2 a - \kappa l_1 u) \frac{\partial g}{\partial u} + (l_1 l_2 b - \kappa l_2 v) \frac{\partial g}{\partial v} = 0,$$
then
$$ l_1 l_2 a - \kappa l_1 u = - \frac{\partial g}{\partial v} \quad \mbox{and} \quad l_1 l_2 b - \kappa l_2 v = \frac{\partial g}{\partial u}.$$
 
Therefore,
$$(l_1 l_2 (a \circ f) - (\kappa \circ f)  l_1 \beta_1 f_1) = - \frac{\partial g}{\partial v} \circ f \quad \mbox{and} \quad  (l_1 l_2 (b\circ f) - (\kappa \circ f)  l_2 \beta_2 f_2)= \frac{\partial g}{\partial u} \circ f$$

Thus,
\begin{align*}
l_1 l_2 (\rho \circ f) &= (l_1 l_2 (a \circ f) ,l_1 l_2 (b \circ f))\\
&= (\kappa \circ f )(l_1 \beta_1 f_1,l_2 \beta_2 f_2) + \left(- \frac{\partial g}{\partial v} \circ f, \frac{\partial g}{\partial u} \circ f\right).
\end{align*}

We have
$$(\kappa \circ f) ( l_1 \beta_1 f_1, l_2 \beta_2 f_2) \in tf(\Theta_X),$$
we will show that 
$$\left(- \frac{\partial g}{\partial v} \circ f,\frac{\partial g}{\partial u} \circ f\right) \in tf(\Theta_X).$$

Since $ \frac{\partial g}{\partial u} \circ f  = C \lambda_f J_2$ and $\frac{\partial g}{\partial v} \circ f = - C  \lambda_f  J_1$, then
$$\left(- \frac{\partial g}{\partial v} \circ f,\frac{\partial g}{\partial u} \circ f\right) = C  \lambda_f  (J_1,J_2) \in tf(\Theta_X).$$

Therefore, $\rho \circ f \in tf(\Theta_X)$ and, consequently,  $$\xi - \tilde{\xi} = df \circ \eta + \rho \circ f \in tf(\Theta_X).$$

Since $t f(\Theta_X) \subseteq t f (\Theta_X) + \omega f(\Theta_2)$, the map is injective and it is surjective.
\end{proof}

By using the coordinates in source and target, we can identify $\Theta(f)$ with $\O_{X,0}^2=\O_{X,0}\oplus\O_{X,0}$. The module $\Theta_{X,0}$ can be seen as a submodule of $\O_{X,0}^n$. With these identifications, the morphism $tf:\Theta_{X,0}\to\Theta(f)$ is the restriction of the map $tf:\O_{X,0}^n\to\O_{X,0}^2$, whose matrix in the canonical basis is the Jacobian matrix of $f$. 

The following result describes the elements of $\ker(\widetilde{ev})$.

\begin{lem} \label{caracterizacao} Let $k = l_2-l_1$. Then,
$$
\ker(\widetilde{ev}) \cong \frac{ \C\left \{ (l_1 \beta_1 t^{r},l_2 \beta_2 t^{r + k}):\ (t^r, t^{r + k}) \in \mathcal{O}^2_{X,0}, r \in \mathbb{N}  \right \} }{ \left \langle (l_1 \beta_1 t^{l_1},l_2 \beta_2 t^{l_1+k}),(l_1 \beta_1 A t^{l_1+\mu_X-1},l_2 \beta_2 A t^{l_1+\mu_X-1+k})  \right \rangle }. 
$$
\end{lem}
\begin{proof}
For each $(a,b)\in \O_{X,0}^2$, we have:
$$\widetilde{ev}([(a,b)])=a \frac{\partial g}{\partial u}\circ f+b \frac{\partial g}{\partial v}\circ f=a l_2 \beta_2^{l_1}\beta_1^{l_2-1} t^{l_1 l_2 - l_1} - b l_1 \beta_1^{l_2}\beta_2^{l_1 -1} t^{l_1 l_2 - l_2}.
$$
Then, $[(a,b)]\in \ker(\widetilde{ev})$ if and only if $(a,b)=c(l_1 \beta_1 t^{s_1},l_2 \beta_2 t^{s_2} )$ such that $l_1 l_2-l_1+s_1=l_1 l_2 - l_2 +s_2$, for some $c\in\O_{X,0}$. It follows that $\ker(\widetilde{ev})$ is generated by monomial pairs of the form
$(l_1 \beta_1 t^r,l_2 \beta_2 t^{r+k})$, for some $r\in\N$ such that $(t^r,t^{r+k})\in\O^2_{X,0}$.

We remember that $tf(\Theta_X)$ is generated on $\mathcal{O}_{X,0}$ by vectors fields
$$\left(\frac{\partial f_1}{\partial {x}_1} h_1, \frac{\partial f_2}{\partial {x}_1}h_1\right),\cdots,\left(\frac{\partial f_1}{\partial {x}_n}h_1,\frac{\partial f_2}{\partial {x}_n}h_1\right),\cdots,$$
$$\left(\frac{\partial f_1}{\partial {x}_1}h_{n-1},\frac{\partial f_2}{\partial {x}_1}h_{n-1}\right),\cdots,\left(\frac{\partial f_1}{\partial {x}_n}h_{n-1},\frac{\partial f_2}{\partial {x}_n}h_{n-1}\right),$$
$$ df \circ \epsilon = (l_1 \beta_1 f_1,l_2 \beta_2 f_2) \quad \mbox{and} \quad df\circ\mathcal H= (-1)^{n+1}( J_1, J_2).$$

Since $\O_{X,0}=\C\{t^{w_1},\cdots,t^{w_n}\}$ and $\O_{Y,0}=\C\{t^{l_1},t^{l_2}\}\subset \O_{X,0}$, we have 
$$(l_1 \beta_1 f_1,l_2 \beta_2 f_2)=(l_1 \beta_1 t^{l_1},l_2 \beta_2 t^{l_1 +k}).$$

We remember that
$$
J_1 = l_1 \beta_1 A  t^{l_1+\mu_X-1} \quad \mbox{and} \quad J_2 = l_2 \beta_2 A t^{l_2+\mu_X-1}
$$

Thus, 
$$df\circ\mathcal H = (-1)^{n+1}(  l_1 \beta_1 A  t^{l_1+\mu_X-1}, l_2 \beta_2 A t^{l_1+\mu_X-1+k}).$$
\end{proof}

\begin{teo}\label{dimensao}
We denote by $s$ the number of elements of the form $(t^{r},t^{r+k})$ in $\mathcal{O}^2_{X,0}$ such that $0\leq r<\mu_X$.
Then, 
$$
\mbox{dim}_{\C} \ker(ev)= l_1 - \delta_X + s -1
$$
and
$$
\mbox{dim}_{\C} \frac{\mathcal{O}_{X,0}}{\langle J_1,J_2 \rangle} =l_{1} + \delta_X + s - 1
$$
\end{teo}
\begin{proof}
By Lemma \ref{nucleo} and  Lemma \ref{caracterizacao},
$$
\ker(ev) \cong \frac{ \C\left \{ (l_1 \beta_1 t^{r},l_2 \beta_2 t^{r + k}):\ (t^r, t^{r + k}) \in \mathcal{O}^2_{X,0}, r \in \mathbb{N}  \right \} }{ \left \langle (l_1 \beta_1 t^{l_1},l_2 \beta_2 t^{l_1+k}),(l_1 \beta_1 A t^{l_1+\mu_X-1},l_2 \beta_2 A t^{l_1+\mu_X-1+k})  \right \rangle }. 
$$

We consider the set $\Gamma= \left\{ l_1+ i: \ i\in\Gamma_X\right\}\cup\left\{l_1+\mu_X-1\right\}$ which is a subset of  $\Gamma_X$.
We denoted by $\Gamma_X \oplus \Gamma_X$ the associated semigroup to $\O^{2}_{X,0}$ e we take the sets:
$$
\Gamma_{k}=\left \{ (r,r+k)\in \Gamma_ X \oplus \Gamma_X \right\} \quad \mbox{and} \quad \Gamma_{\Theta}=\left\{ (l,l+k):\ l \in \Gamma \right\} \subset \Gamma_{k}
$$

We can identify the elements of $\ker(ev)$ with the elements of $\Gamma_{k} \setminus \Gamma_{\Theta}$.
Therefore, we compute the number of elements in $\Gamma_{k} \setminus \Gamma_{\Theta}$.

We observe that $\Gamma_{k}$ is equal to
\begin{align*}
\underbrace{\left\{(r,r+k):\ 0 \leq r < \mu_X \right\}}_{s \ \text{elements}} \cup \underbrace{\left\{(\mu_X,\mu_X+k),(\mu_X+1,\mu_X+1+k), \cdots \right\}}_{\text{all elements}}
\end{align*}
and $\Gamma_{\Theta}$ is 
\begin{align*}
&\overbrace{\left\{ (l_1+p,l_1+p+k):\ 0 \leq p < \mu_X,\ p\in\Gamma_X \right\}}^{\delta_X \ \mbox{elements}}\\
 \cup & \underbrace{\left\{ (l_{1}+ \mu_X - 1,l_1+ \mu_X - 1 +k), (l_1 + \mu_X,l_1+\mu_X+k),\cdots\right\}}_{\text{all elements}}.
\end{align*}

Then, we describe $\Gamma_{k} \setminus \Gamma_{\Theta}$ as 
\begin{align*}
\overbrace{\left\{(r,r+k):\ 0 \leq r < \mu_X \right\}}^{s \ \text{elements}}\cup\overbrace{\left\{(\mu_X,\mu_X+k),\cdots, (l_1+\mu_X-2,l_1+\mu_X-2+k)\right\}}^{l_1-1 \ \text{elements}} \\
\end{align*}
\begin{align*}
\setminus \underbrace{\left\{ (l_1+p,l_1+p+k):\ 0 \leq p < \mu_X,\  p\in\Gamma_X \right\}}_{\delta_X \ \mbox{elements}}
\end{align*}

Therefore, $$\mbox{dim}_{\C} \ker(ev) = s +  l_1 - 1 - \delta_X = l_1 - \delta_X + s - 1.$$

Now, we compute the dimension of $\O_X / \langle J_1,J_2 \rangle$.
We consider the sets 
$$
\Gamma_{J_1}= \left\{ l_1+\mu_X-1 + i: \ i\in\Gamma_X\right\}
$$ 
and
$$
\Gamma_{J_2}= \left\{ l_1+\mu_X-1 + k+ j: \ j \in\Gamma_X\right\}
$$
which are subset of  $\Gamma_X$. 
We can identify the elements of $\O_X / \langle J_1,J_2\rangle$ with the elements of $\Gamma_{X} \setminus (\Gamma_{J_1} \cup \Gamma_{J_2})$.
Therefore, we compute the number of elements in $\Gamma_{X} \setminus (\Gamma_{J_1} \cup \Gamma_{J_2})$.

We observe that there exists common elements in $\Gamma_{J_1}$ and $\Gamma_{J_2}$, we need to know the elements.
We have that all the elements after $l_1 + \mu_X -1 + k+ \mu_X$ are common elements in $\Gamma_{J_1}$ and $\Gamma_{J_2}$. 
Then, we consider the subsets 
$$
A=\left \{ l_1 + \mu_X - 1 + i :\ 0\leq i \leq \mu_X + k, i \in \Gamma_X \right \}
$$ 
and 
$$
B=\left \{ l_1 + \mu_X -1 + k + j :\ 0\leq j < \mu_X , j \in \Gamma_X \right \}.
$$
Given $r$ and $r+k$ in $\Gamma_X$ such that $0 \leq r < \mu_X$, the elements common in $A$ and $B$ are the elements of form:
$$l_1 + \mu_X -1 +(r+k) = l_1 + \mu_X - 1 + k + r.$$ 
Since $s$ is the number of elements $r \in \Gamma_X$, $0 \leq r < \mu_X$, such that $r+k \in \Gamma_X$, then $A\cap B$ contained $s$ common elements.
We have that $A$ contained $\delta_X +  k + 1$ elements and $B$ contained $\delta_X$ elements, thus $A\cup B$ contained $\mu_X + k + 1 - s$ elements.

Thus, $\Gamma_{X} \setminus (\Gamma_{J_1} \cup \Gamma_{J_2})$ is equivalent to
\begin{align*} 
\overbrace{\left\{j \in \Gamma_X; 0\leq j < \mu_X\right\}}^{\delta_X \ \mbox{elements}}\cup\overbrace{\left\{\mu_X,\cdots,l_1+\mu_X + \mu_X + k - 1 \right\}}^{l_1 + \mu_X + k  \ \mbox{elements}}
\end{align*}
\begin{align*}
\setminus \underbrace{\left \{ l_1 + \mu_X - 1 + i :\ 0\leq i \leq \mu_X  \right \}\cup \left \{ l_1 + \mu_X -1+ k + j :\ 0\leq j < \mu_X \right \}}_{\mu_X + k + 1 - s  \ \mbox{elements}}.
\end{align*}
Therefore,
$$\mbox{dim}_{\C} \frac{\mathcal{O}_{X,0}}{\langle J_1,J_2 \rangle} = \delta_X + l_1 + \mu_X + k - (\mu_X+ k + 1 - s) = l_1 + \delta_X + s -1.$$
\end{proof}

\begin{cor}\label{corollarydimensions} Let $(X,0)\subset(\C^n,0)$ be an irreducible curve ICIS, weighted homogeneous and let $f:(X,0)\to(\C^2,0)$ be a finite map germ of degree 1 onto its image $(Y,0)$ and $f$ is consistent with $(X,0)$. Then,
$$ \A_{e}\mbox{-}\codim(f) =  \dim_{\C} \frac{\mathcal{O}_{X,0}}{\langle J_1,J_2 \rangle} -\mu_X + \dim_{\C} \frac{J_{g}\mathcal{O}_{X,0}}{J_{g}\mathcal{O}_{Y,0}}.$$
\end{cor}
\begin{proof} By Lemma \ref{sobrejetora},
$$ \A_{e}\mbox{-}\codim(f) =  \dim_{\C} \ker(ev) + \dim_{\C} \frac{J_{g}\mathcal{O}_{X,0}}{J_{g}\mathcal{O}_{Y,0}}.$$

By Theorem \ref{dimensao},
$$\dim_{\C} \frac{\mathcal{O}_{X,0}}{\langle J_1,J_2 \rangle} - \dim_{\C} \ker(ev)= l_{1} + \delta_X + s - 1- \left( l_1 - \delta_X + s -1 \right)  = \mu_X.$$

Therefore,
$$ \A_{e}\mbox{-}\codim(f) = \dim_{\C} \frac{\mathcal{O}_{X,0}}{\langle J_1,J_2 \rangle} -\mu_X  + \dim_{\C} \frac{J_{g}\mathcal{O}_{X,0}}{J_{g}\mathcal{O}_{Y,0}}.$$
\end{proof}

\section{The image Milnor number}

The first and second authors in \cite{DJ}, define the image Milnor number $\mu_I(f)$ of $f:(X,0)\to(\C^2,0)$, where $(X,0)$ is a plane curve and $f$ is $\A$-finite.  
In this work, naturally, we generalize the concept of the image Milnor number to the case that $(X,0)$ is a space curve ICIS.  
We consider $(X,0)\subset(\C^n,0)$ a curve ICIS and $f:(X,0)\to(\C^2,0)$ be a finite map germ of degree 1 onto its image $(Y,0)$.
We suppose $f$ has finite singularity type, it follows from Theorem \ref{estabilizacao} that $f$ always admits a stabilisation $F: (\mathcal{X},0) \to (\C \times \C^2,0)$, given by $F(s,x)=(s,f_s(x))$. This means that $f_s:X_s\to B_\epsilon$ is stable, for all $s\ne0$ small enough, where $B_\epsilon$ is a small enough ball centered at the origin in $\C^2$.  Thus the image $Y_s=f_s(X_s)$ has the homotopy type of a wedge of 1-spheres.

\begin{defi} The \emph{image Milnor number} $\mu_I(f)$ is 
the number of 1-spheres in $Y_s$.
\end{defi}

Analogous to case where $(X,0)$ is a plane curve \cite{DJ}, $\mu_I(f)$ is well-defined, that is, it is independent of the stabilisation and of the representative. The image Milnor number satisfies the equality.
$$\mu_{I}(f) = \frac{\mu(Y,0) + \mu(X,0)}{2}.$$

For the next results, we use the definition of the delta invariant of $f$, $\delta(f)$, which was introduced in \cite{NT} for maps  of degree 1 between curves.

\begin{defi}[\cite{NT}] Let $f:(X,0)\to(Y,0)$ be a holomorphic map of degree 1 between curves $(X,0)$ and $(Y,0)$. The \emph{delta invariant} of $f$ is
$$
\delta(f)=\dim_\C\frac{\O_{X,0}}{f^* \O_{Y,0}}.
$$
\end{defi}

This number also satisfies $\delta(Y,0)=\delta(X,0)+\delta(f)$ when $(X,0)$ and $(Y,0)$ are irreducible curves.
Therefore, when $(X,0)$ and $(Y,0)$ are irreducible curves  $\mu(Y,0)-\mu(X,0)=2\delta(f)$ (see \cite{NT}), thus
$$
\mu_I(f)=\mu(Y,0) - \delta(f)=\mu(X,0)+\delta(f).
$$

Now we consider again, $(X,0)\subset(\C^n,0)$ an irreducible curve ICIS and weighted homogeneous of type $(w_1,\cdots,w_n;d_1,\cdots,d_{n-1})$, and $f:(X,0)\to(\C^2,0)$ a finite map germ of degree 1 onto its image $(Y,0)$, such that $f$ is consistent with $(X,0)$. 

We take $\mathcal{C}$ the ideal generated by $\lambda_f$ in $\O_{X,0}$. Then we have the following result.

\begin{prop}\label{delta}
Let $(X,0)\subset(\C^n,0)$ be an irreducible curve ICIS and weighted homogeneous of type $(w_1,\cdots,w_n;d_1,\cdots,d_{n-1})$ and let $f:(X,0)\to(\C^2,0)$ be a finite map germ of degree 1 onto its image $(Y,0)$ and consistent with $(X,0)$.
Then, 
$$
\dim_{\C}\frac{\O_{X,0}}{\mathcal{C}} = 2\delta(f) \quad \mbox{and} \quad \dim_{\C}\frac{\O_{Y,0}}{\mathcal{C}} = \delta(f).
$$
Furthermore, we have that $\frac{\mathcal{O}_{X,0}}{\langle J_1,J_2 \rangle}$ is isomorphic to $\frac{\mathcal{C}}{J_g \mathcal{O}_{X,0}}$.
\end{prop}
\begin{proof} 
By Lemma \ref{lemmalambda},
$$
\lambda_f=t^{\mu_Y-\mu_X}=t^{2\delta(f)}.
$$
Therefore, the weighted degree of $\lambda_f$ with the weights $w_1,\cdots,w_n$ is $2\delta(f)$. 
Denoted by $(h_1,\cdots,h_{n-1}):(\C^n,0)\to (\C^{n-1},0)$ the map that generates $(X,0)$,  by Bezout theorem, we obtain
\begin{align*}
\dim_{\C}\frac{\O_{X,0}}{\mathcal{C}} &= \dim_{\C}\frac{\O_{X,0}}{\langle \lambda_f \rangle} = \dim_{\C}\frac{\O_{2}}{\langle h_1,\cdots,h_{n-1}, \lambda_f \rangle}\\
&=\frac{d_1\cdots d_{n-1}2\delta(f)}{w_1 \cdots w_n}=2\delta(f).
\end{align*}

Therefore,
$$
\dim_{\C}\frac{\O_{X,0}}{\mathcal{C}} = 2\delta(f).
$$

To obtain the second equality, we first show that $\mathcal{C} \subset \O_{Y,0}$. We recall that $\mathcal{C}$ is the ideal generated by $\lambda_f=t^{2\delta(f)}$ in $\O_{X,0}$. 
Then, we need to show that $\sigma t^{2\delta(f)} \in \O_{Y,0}$, for all $\sigma \in \O_{X,0}$.
To see this, we consider the numerical semigroups $\Gamma_X$ and $\Gamma_Y$ associated the curves $(X,0)$ and $(Y,0)$ respectively. We recall that $\Gamma_X$ and $\Gamma_Y$ are symmetric. 

We show that if $2\delta(f) \in \Gamma_Y$, then $t^{2\delta(f)} \in \O_{Y,0}$. 
In fact,  since $\Gamma_Y$ is symmetric we have $2\delta(f) \in \Gamma_Y$ or $\mu_Y-1-2\delta(f) \in \Gamma_Y$.
We suppose $\mu_Y-1-2\delta(f) \in \Gamma_Y$, then
$$
\mu_Y-1-2\delta(f)= \mu_Y - 1 - \mu_Y + \mu_X = \mu_X-1 \in \Gamma_Y \subset \Gamma_X.
$$
But $\mu_X-1 \notin \Gamma_X$, because $\mu_X$ is the conductor of $\Gamma_X$. 
Therefore, $2\delta(f) \in \Gamma_Y$, thus $t^{2\delta(f)} \in \O_{Y,0}$.

Now, let $\sigma=t^{a} \in \O_{X,0}$ with $a \in \Gamma_X$, then $2\delta(f)+a \in \Gamma_Y$. 
In fact, we have $2\delta(f)+a \in \Gamma_Y$ or $\mu_Y-1-(2\delta(f)+a) \in \Gamma_Y$.
If $\mu_Y-1-(2\delta(f)+a) \in \Gamma_Y$, then
$$
\mu_Y-1-\mu_Y+\mu_X- a= \mu_X-1-a \in \Gamma_Y \subset \Gamma_X.
$$
Since $\Gamma_X$ is symmetric, we have $a \notin \Gamma_X$, but $\sigma=t^{a} \in \O_{X,0}$. 
Thus, $\mu_Y-1-(2\delta(f)+a) \notin \Gamma_Y$ and  consequently $2\delta(f)+a \in \Gamma_Y$.
Therefore, $\mathcal{C} \subset \O_{Y,0}$.
We obtain,
$$
\dim_{\C}\frac{\O_{Y,0}}{\mathcal{C}}=\dim_{\C}\frac{\O_{X,0}}{\mathcal{C}} - \dim_\C\frac{\O_{X,0}}{f^* \O_{Y,0}} = 2\delta(f) -\delta(f) = \delta(f).
$$

To complete the proof, we observe that multiplication by $\lambda_f$ gives an isomorphism $\phi:\O_{X,0}\to\mathcal C$. Moreover, $\phi(\langle J_1,J_2 \rangle)=J_g \mathcal{O}_{X,0}$ and hence, it induces an isomorphism
$$\frac{\mathcal{O}_{X,0}}{\langle J_1,J_2 \rangle} \longrightarrow \frac{\mathcal{C}}{J_g \mathcal{O}_{X,0}}.$$
\end{proof}

\begin{teo}\label{teoprincipal}
Let $(X,0)\subset(\C^n,0)$ be an irreducible weighted homogeneous curve ICIS and let $f:(X,0)\to(\C^2,0)$ be a finite map germ of degree 1 onto its image $(Y,0)$ and consistent with $(X,0)$.
Then, 
$$\A_{e}\mbox{-}\codim(X,f) = \mu_{I}(f).$$
Moreover, let $\alpha:(\C,0)\to(X,0)$ be a parametrisation of $(X,0)$. Then, 
$$
\mu_{I}(f)=\delta(X,0)+\mu_{I}(f\circ\alpha)
$$
and
$$
\A_e\mbox{-}\codim(f)=\A_e\mbox{-}\codim(f\circ\alpha)-\frac{1}{n-1}\A_e\mbox{-}\codim(\alpha)
$$
\end{teo}
\begin{proof}
By Corollary \ref{corollarydimensions} and by Proposition \ref{delta}, we obtain
\begin{align*}
\A_{e}\mbox{-}\codim(f) &= \dim_{\C} \frac{\mathcal{O}_{X,0}}{\langle J_1,J_2 \rangle}  -  \mu_X + \dim_{\C} \frac{J_{g}\mathcal{O}_{X,0}}{J_{g}\mathcal{O}_{Y,0}}\\
&= \dim \frac{\mathcal{C}}{J_g \mathcal{O}_{X,0}} - \mu_X + \dim_{\C} \frac{J_{g}\mathcal{O}_{X,0}}{J_{g}\mathcal{O}_{Y,0}}.
\end{align*}

We have the exact sequence, 
$$ 0 \longrightarrow \frac{J_{g}\mathcal{O}_{X,0}}{J_{g}\mathcal{O}_{Y,0}} \longrightarrow \frac{\mathcal{C}}{J_g \mathcal{O}_{Y,0}}\longrightarrow \frac{\mathcal{C}}{J_g \mathcal{O}_{X,0}} \longrightarrow 0,$$
thus, by the exact sequence and by Proposition \ref{delta},
\begin{align*}
\A_{e}\mbox{-}\codim(f) &=\dim_{\C} \frac{\mathcal{C}}{J_g \mathcal{O}_{Y,0}} - \mu_X\\
&=\dim_{\C} \frac{\O_{Y,0}}{J_g \mathcal{O}_{Y,0}}  - \dim_{\C} \frac{\O_{Y,0}}{\mathcal{C}} - \mu_X \\
&= \mu_Y - \delta(f) - \mu_X = \delta(f).
\end{align*}

Since $(X,0)$ is weighted homogeneous, we have $\mu_X=\tau(X,0)$, hence
$$\A_{e}\mbox{-}\codim(X,f) = \A_{e}\mbox{-}\codim(f) + \mu_X =  \delta(f) + \mu_X = \mu_I(f).$$

To prove the second part, we observe that $f\circ\alpha:(\C,0)\to(Y,0)$ is a parametrisation of $(Y,0)\subset(\C^2,0)$, then by \cite{M2}, 
$$
\A_e\mbox{-}\codim(f\circ\alpha)=\mu_{I}(f\circ\alpha)=\delta(Y,0).
$$
By \cite{NT}, $\delta(Y,0)=\delta(X,0)+\delta(f)$, and by definition of image Milnor number, we obtain
\begin{align*}
\mu_I(f)&=\mu(X,0)+\delta(f)=\mu(X,0)+\delta(Y,0)-\delta(X,0)\\
&=\delta(X,0)+\delta(Y,0)=\delta(X,0)+\mu_I(f\circ\alpha).
\end{align*}

We also have,
\begin{align*}
\A_e\mbox{-}\codim(f)&= \delta(f)= \delta(Y,0) - \delta(X,0) \\
&=\A_e\mbox{-}\codim(f\circ\alpha)-\delta(X,0).
\end{align*}
It follows from \cite{HHR} that  $\A_e\mbox{-}\codim(\alpha)=(n-1) \delta(X,0)$, thus
\begin{align*}
\A_e\mbox{-}\codim(f)&= \A_e\mbox{-}\codim(f\circ\alpha)-\frac{1}{n-1}\A_e\mbox{-}\codim(\alpha).
\end{align*}
\end{proof}

We observe that, in particular, if $(X,0)$ is a plane curve, by \cite{M2}, $\mu_I(\alpha)=\delta(X,0)$ and we obtain 
$$
\mu_{I}(f)=\mu_{I}(\alpha)+\mu_{I}(f\circ\alpha).
$$

\begin{exam}\label{example1}
We consider $(X,0)$ a weighted homogeneous  curve defined by $h(x,y) = x^{2}-y^{3}$  and $f:(X,0) \to (\C^{2},0)$ a map defined by $f(x,y)=(x,y^2)$.
Then, $(Y,0)$ is weighted homogeneous plane curve defined by $g(u,v) = u^{4}-v^{3}$.

We have, $\mu(X,0) =2$, $\mu(Y,0)=6$, then 
$$\mu_{I}(f)=(\mu(Y,0) + \mu(X,0))/2 = (6 + 2)/2 = 4.$$
By \cite{MM}, 
$$ \A_{e}\mbox{-}\codim(X,f)=\dim_{\C} \frac{\Theta(\gamma)}{t\gamma(\Theta_{2})+\gamma^{*}(\Derlog D(G))},
$$
where $G:\C^{3} \to \C^4$ defined by $G(x,y,s)=(x,y^2,x^2-y^3+sy,s)$ is a stability of map $(h,\bar{f}): \C^2 \to \C^3$, given by $(h,\bar{f})(x,y)=(x^2-y^3,x,y^2)$.
We observe that, $D(G)=Im(G)$, because $G$ is no surjective, then $\Derlog D(G)=\Derlog \mbox{Im}(G)$.
Hence, to compute $ \A_{e}\mbox{-}\codim(X,f)$, we find $H: \C^4 \to \C$ such that $\mbox{Im}(G)=H^{-1}(0)$. We do this using the Singular software,
$$H(z,u,v,w)= z^2-v^3-2u^2z+2v^2w-vw^2+u^4.$$
We use the map $H$  to compute $\Derlog \mbox{Im}(G)$ and then we can calculate
$$\dim_{\C} \frac{\Theta(\gamma)}{t\gamma(\Theta_{2})+\gamma^{*}(\Derlog \mbox{Im}(G))}.$$ 
Therefore,
$$ \A_{e}\mbox{-}\codim(X,f)=\dim_{\C} \frac{\Theta(\gamma)}{t\gamma(\Theta_{2})+\gamma^{*}(\Derlog \mbox{Im}(G))} = 4 = \mu_{I}(f). $$
and
$$\A_{e}\mbox{-}\codim(f)= \A_{e}\mbox{-}\codim(X,f)-\mu(X,0)=4-2=2=\delta(f).$$

Now, we consider $\alpha:(\C,0)\to(X,0)$, given by $\alpha(t)=(t^3,t^2)$, the parametrisation of $(X,0)$. Then,
$\A_{e}\mbox{-}\codim(\alpha)=\mu_I(\alpha)=\delta(X,0)=1$ and $\A_{e}\mbox{-}\codim(f\circ\alpha)=\mu_I(f\circ\alpha)=\delta(Y,0)=3$.
Therefore,
$$
\mu_I(f)=\mu_I(\alpha)+\mu_I(f\circ\alpha).
$$
and
$$
\A_{e}\mbox{-}\codim(f)= \A_{e}\mbox{-}\codim(f\circ\alpha)- \A_{e}\mbox{-}\codim(\alpha).
$$
\end{exam}

\begin{exam} 
We consider $(X,0)\subset (\C^3,0)$ a weighted homogeneous curve defined by $h(x,y,z) = (x^{3}-y^{2},xy-z)$ with parametrization $\alpha:(\C,0)\to (\C^3,0)$ defined by $\alpha(t)=(t^2,t^3,t^5)$  and $f:(X,0) \to (\C^{2},0)$ a map defined by $f(x,y,z)=(x,y^2z)$, thus $f$ is consistent with $(X,0)$. Then, $(Y,0)$ is a weighted homogeneous plane curve defined by $g(u,v) = u^{11}-v^{2}$ with parametrization $\beta:(\C,0)\to (\C^2,0)$ defined by $\beta(t)=(t^2,t^{11})$. 

We use the Singular software to calculate the Milnor number, we have $\mu(X,0)=2$ and $\mu(Y,0)=10$.
Thus, $\mu_{I}(f)=(\mu(Y,0) + \mu(X,0)) / 2 = (10 + 2) / 2 = 6$.

By \cite{MM}, 
$$ \A_{e}\mbox{-}\codim(X,f)=\dim_{\C} \frac{\Theta(\gamma)}{t\gamma(\Theta_{2})+\gamma^{*}(\Derlog D(G))},
$$
where $G:\C^{4} \to \C^5$ defined by $G(x,y,z,s)=(x^3-y^2,xy-z,x,y^2z+sy,s)$ is a stability of map $(h,\bar{f}): \C^3 \to \C^4$, given by $(h,\bar{f})(x,y,z)=(x^3-y^2,xy-z,x,y^2z)$.
As in example \ref{example1},  we use the Singular software, thus,
$$ \A_{e}\mbox{-}\codim(X,f)=\dim_{\C} \frac{\Theta(\gamma)}{t\gamma(\Theta_{2})+\gamma^{*}(\Derlog \mbox{Im}(G))} = 6=\mu_{I}(f).$$
\end{exam}

\begin{exam}
We consider $(X,0)\subset (\C^3,0)$ a weighted homogeneous curve defined by $h(x,y,z) = (x^{5}-y^{2},xy-z)$ with parametrization $\alpha:(\C,0)\to (\C^3,0)$ defined by $\alpha(t)=(t^2,t^5,t^7)$  and $f:(X,0) \to (\C^{2},0)$ a map defined by $f(x,y,z)=(x^2,y+z)$. We observe that $f$ is not consistent with $(X,0)$. Then, $(Y,0)$ is a plane curve defined by $g(u,v) = u^{7}-2u^6 +4u^3 v^2 + u^5-v^{4}$ with parametrization $\beta:(\C,0)\to (\C^2,0)$ defined by $\beta(t)=(t^4,t^5 + t^7)$. 

We use the Singular software to calculate the Milnor number and the Tjurina number, we have $\mu(X,0)=4$, $\tau(X,0)=4$, $\mu(Y,0)=12$ and $\tau(Y,0)=11$.
Thus, $(Y,0)$ is not weighted homogeneous. We also have, $\mu_{I}(f)=(\mu(Y,0) + \mu(X,0)) / 2 = (12 + 4) / 2 = 8$.

By \cite{MM}, 
$$ \A_{e}\mbox{-}\codim(X,f)=\dim_{\C} \frac{\Theta(\gamma)}{t\gamma(\Theta_{2})+\gamma^{*}(\Derlog D(G))},
$$
where $G:\C^{4} \to \C^5$ defined by $G(x,y,z,s)=(x^5-y^2+sx,xy-z,x^2,y+z,s)$ is a stability of map $(h,\bar{f}): \C^3 \to \C^4$, given by $(h,\bar{f})(x,y,z)=(x^5-y^2,xy-z,x^2,y+z)$.
As in example \ref{example1},  we use the Singular software, thus,
$$ \A_{e}\mbox{-}\codim(X,f)=\dim_{\C} \frac{\Theta(\gamma)}{t\gamma(\Theta_{2})+\gamma^{*}(\Derlog \mbox{Im}(G))} = 7.$$
Therefore,
$$ \A_{e}\mbox{-}\codim(X,f)<\mu_{I}(f).$$
We observe that the hypotheses $f$ consistent with of $(X,0)$ is necessary to obtain an equality.
\end{exam}

\bibliographystyle{amsplain}
\bibliography{refbiblio}

\end{document}